\documentclass[a4paper]{amsart} 

\usepackage{amsmath,amsthm}
\usepackage{amssymb}

\usepackage{graphicx}

\usepackage[all]{xy}
\usepackage{tikz} 
\usepackage[T1]{fontenc}
\newtheorem{theorem}{Theorem}[section]
\newtheorem{proposition}[theorem]{Proposition}

\newtheorem{corollary}[theorem]{Corollary}

\theoremstyle{definition}

\renewcommand{\[}{\begin{equation}}
\renewcommand{\]}{\end{equation}}
\newcommand{\wegengruen}{\end{equation}}

 \newcommand{\R}{\mathbb{R}}
 \newcommand{\C}{\mathbb{C}}
 \newcommand{\N}{\mathbb{N}}
 \newcommand{\Z}{\mathbb{Z}}
 \newcommand{\D}{\mathbb{D}}

 \newcommand{\A}{\mathcal{A}}
  \newcommand{\cC}{\mathcal{C}}
    \newcommand{\cS}{\mathcal{S}}
  \newcommand{\K}{\mathcal{K}}
 \newcommand{\hH}{\mathcal{H}}
 \newcommand{\T}{\mathcal{T}}

\newcommand{\hsp}{{\hspace{-1pt}}}
\newcommand{\hs}{{\hspace{1pt}}}
 
\newcommand{\lra}{\longrightarrow}

\newcommand{\ra}{\rightarrow}

\newcommand{\dd}{\mathrm{d}}
\newcommand{\im}{\mathrm{i}}
\newcommand{\id}{\mathrm{id}}
\newcommand{\e}{\mathrm{e}}
\newcommand{\Mat}{\mathrm{Mat}}

\newcommand{\lN}{\ell_2(\N)}

\newcommand{\veem}{\mbox{$\bigvee\hspace{-7.2pt}{}^{\text{-}}{}^{\;2g} \hs $}}
\newcommand{\veep}{\mbox{$\bigvee\hspace{-9pt}{}^{+}{}^{\,g} \hs$}}

\def\bD{\bar\D}
\def\LD{L_2(\D)}
\def\AD{A_2(\D)}
\def\KAD{K(A_2(\D))}
\def\KlN{K(\lN)}

\def\CD{C(\bar\D)}
\def\dom{\mathrm{dom}}
\def\dim{\mathrm{dim}}
\def\spec{\mathrm{spec}}

\def\ind{\mathrm{ind}}
\def\cspan{\overline{\mathrm{span}}}

\def\ot{\otimes} 
\def\bS{\mathbb{S}^1} 
\def\bT{\mathbb{T}} 
 
\def\CTq{C(\mathbb{T}_{g,q})} 
\def\Cinf{\cC^\infty(\mathbb{T}_{g,q})} 
\def\CPq{C(\mathbb{P}_{g,q})}
\def\s{\sigma} 
\def\a{\alpha} 
\def\g{\gamma} 
\def\ker{\mathrm{ker}} 
\def\CS{C(\mathbb{S}^1)} 

\newcommand{\ip}[2]{\langle{#1},{#2}\rangle} 
\newcommand{\ipp}[2]{\langle\hsp\langle{#1},{#2}\rangle\hsp\rangle} 
\newcommand{\msum}[2]{\underset{#1}{\overset{#2}{\mbox{$\sum$}}}}

\title{A spectral triple for noncommutative compact surfaces}

\author{Fredy D\'iaz Garc\'ia}

\address{Instituto de F\'isica y Matem\'aticas\\
Universidad Michoacana de San Nicol\'as de Hidalgo\\ 
Cd.~Universitaria,  Edificio C-3\\
58040 Morelia, Michoac\'an\\
M\'exico} 

\email{lenonndiaz@gmail.com}

\author{Elmar Wagner}

\address{Instituto de F\'isica y Matem\'aticas\\
Universidad Michoacana de San Nicol\'as de Hidalgo\\ 
Cd.~Universitaria,  Edificio C-3\\
58040 Morelia, Michoac\'an\\
M\'exico} 

\email{elmar@ifm.umich.mx}



\dedicatory{Dedicated to Paul Baum on the occasion of 
his 80th  birthday who taught me like nobody else 
that teaching difficult mathematics doesn't have to be difficult and can be fun.}

\usepackage{geometry}
\begin{document}

\begin{abstract}
A Dirac operator is presented that will yield a $1^+$-summable regular even spectral triple 
for all noncommutative compact surfaces defined as subalgebras of the Toeplitz algebra. 
Connes' conditions for noncommutative spin geometries are analyzed and it is argued 
that the failure of some requirements is mainly due to a wrong choice of a noncommutative spin bundle. 
\end{abstract}

%%% ----------------------------------------------------------------------
\maketitle
%%% ----------------------------------------------------------------------

\section{Introduction} 
One of the most fundamental notions of  Connes' noncommutative geo\-me\-try \cite{C} is that of a spectral triple. 
An extensive amount of research has been dedicated to finding  novel examples. However, proving that 
a spectral triple satisfies all requirements of a noncommutative spin geometry can be a difficult task. 
Among the noteworthy examples are the noncommutative torus which satisfies all proposed conditions \cite{C1}, 
the isospectral deformations of toric manifolds \cite{CL}, and the standard Podle\'s sphere \cite{DS}. In the latter case,  
the original conditions had to be modified to conform with the basic principles of a noncommutative spin geometry \cite{W}.  
This seems to be a common feature for examples arising in quantum group theory. In particular, for the natural 
spectral triple on quantum SU(2)~\cite{DLSSV}, it was \textit{proven} that the requirement of a real structure 
demands a modification of original framework. A proposal for such a modification can be found in~\cite{BCDS}. 

In this paper, we present a $1^+$-summable regular even spectral triple 
for the noncommutative compact surfaces from \cite{W0} and study it in the context 
of Connes' noncommutative spin geometry \cite{GFV}. It turns out that the same Dirac operator can be used for all  
noncommutative compact surfaces including the non-orientable ones. This already indicates that these spectral triples are not very useful 
for the calculation of topological invariants (yet this could be a welcome effect: if all quantum spaces behaved like the classical ones, 
there were no need for a quantization). The reason is that the spectral triples are defined on the ``wrong'' Hilbert space, 
namely on the holomorphic functions of the Bergman space on the unit disc rather than on sections of the noncommutative spin bundle. 
Although this seems obvious, such effects could easily be disguised in more complicated examples, especially if the finiteness and the regularity 
conditions can be satisfied as in our case. Under more favorable conditions, the finiteness and the regularity 
conditions lead the correct module related to the spin bundle. Here, however, 
as a consequence of working on the ``wrong'' Hilbert space,  the spectrum of the Dirac operator 
resembles that of a 1-dimensional spectral triple and so does a possible real structure. Furthermore, the first order condition can only be satisfied 
up to compacts whereas the orientation condition and the Poincar\'e duality fail. In this sense, our paper is rather a warning that 
modifying or not fulfilling some axiomatic conditions may have substantial effects on the noncommutative spin geometry.

\section{Noncommutative compact surfaces}    \label{sec-1} 

The noncommutative compact surfaces of any genus \cite{W0} will be defined as subalgebras of the continuous functions on the quantum disc.  
As explained in \cite{KL}, the universal C*-algebra of the quantum disc can be represented by the Toeplitz algebra. For a description of the Toeplitz algebra, 
consider the open complex unit disc $\D:= \{ z\in \C: |z|<1\}$ with the standard Lebesgue measure  and let $\bD:= \{ z\in \C: |z|\leq 1\}$ denote 
its closure in $\C$. We write $\LD$ for  the Hilbert space of square-integrable functions and  $\AD$ for the closed subspace 
of holomorphic functions  on $\D$. Let $\hat P$ denote the orthogonal projection from $\LD$  onto $\AD$. 
The Toeplitz operator $\hat T_f\in B(\AD)$ with symbol $f\in \CD$ is defined by 
$$
 \hat T_f(\psi):= \hat P(f\hs \psi), \qquad \psi \in \AD \subset \LD, 
$$
and the Toeplitz algebra $\T$ is the C*-algebra generated by all $\hat T_f$ in $B(\AD)$.  
As is well known \cite{V},  the compact operators $\KlN \cong \KAD$ 
belong to $\T$ and that the quotient $\T / \KlN \cong \CS$ gives rise to 
the C*-algebra extension 
\[     \label{Cex1} 
\xymatrix{
 0\;\ar[r]&\; \KlN\;\ar[r] &\; \T \;\ar[r]^{\hat \s\ \ } &\; \CS\;\ar[r] &\ 0\, ,} 
 \]
with the symbol map $\hat \s : \T \lra \CS$ given by $\hat\s(\hat T_f) =  f\!\! \upharpoonright_{\mathbb{S}^1}$ 
for all $f\in \CD$. 

Recall that each closed surface can be constructed from a convex polygon by a suitable identification of its edges.  
Instead of edges of a polygon, we will consider arcs on the boundary of the unit disc. In this manner, the C*-algebra of continuous 
functions on a closed surface can be viewed as a subalgebra of $\CD$. Identifying points on the boundary means that the 
functions belonging to the specified subalgebra must have the same values on identified points. In correspondence with the presentation  
of oriented closed surfaces, let $g\in \N$ and define  $4g$ arcs on the circle $\bS$ by 
\begin{align*}
  a_k,\, a_k^{-1} : \,[0,1]\, \lra\, \bS,\quad  a_k(t) := \e^{\pi\im\frac{k-1+t}{2g} } , \quad a_k^{-1}(t) := \e^{\pi\im\frac{2g+ k-t}{2g} },\quad k=1,\ldots,2g. 
\end{align*}
By standard operations from algebraic topology, one can readily show that 
$$
 \bT_g:= \bar\D\,/\sim\,,\quad z\sim z\ \,\text{and}\ \,a_k(t)\sim a_k^{-1}(t) \ \, \text{for \,all} \ \, z\in\D, \ \,t\in[0,1], \  \,k=1,\ldots,2g, 
$$
is homeomorphic to a closed oriented surface of genus $g$. 

Viewing the symbol map $\hat \s :\T \ra \CS$ as the counterpart of an embedding of the circle into the quantum disc, we define
\[   \label{CTq}
\CTq := \{ f\in\T: \hat\s(f)(a_k(t))= \hat\s(f)(a_k^{-1}(t)) \ \,\text{for \,all}  \ \,t\in[0,1], \  \,k=1,\ldots,2g\}, 
\] 
see  \cite{W0}. To include the 2-sphere with genus 0, we consider additionally 
\begin{align*}
  a_0,\, a_0^{-1} : \,[0,1]\, \lra\, \bS,\quad  a_0(t) := \e^{\pi\im t } , \quad a_0^{-1}(t) := \e^{-\pi\im t } . 
\end{align*}
and define $C(\mathbb{S}^2_q):= C(\mathbb{T}_{0,q})$ as in \eqref{CTq} with $k=0$. 
Similarly, for non-orientable closed surfaces, let 
\begin{align}  \label{ab} 
  a_k,\, b_k : \,[0,1]\, \lra\, \bS,\quad  a_k(t) := \e^{\pi\im\frac{k-1+t}{g} } , \quad b_k(t) := \e^{\pi\im\frac{-k+t}{g} },\quad k=1,\ldots,g, 
\end{align}
 and set 
 \[   \label{CPq}
\CPq := \{ f\in\T: \hat\s(f)(a_k(t))= \hat  \s(f)(b_k(t)) \ \,\text{for \,all}  \ \,t\in[0,1], \  \,k=1,\ldots,g\}. 
\] 
From the continuity of the symbol map,  it follows that $\CTq$ and $\CPq$ are C*-subalgebras of $\T$. 

 An alternative, albeit less illustrative description of these C*-algebras can be given by considering $L_2(\bS)$ 
 with orthonormal  basis $\{ e_k:= \frac{1}{\sqrt{2\pi}}u^k:k\in\Z\}$, 
 where $u(\e^{\im t}) =  \e^{\im t}$ denotes the unitary generator of $C(\bS)\subset L_2(\bS)$. 
 Let $P$ denote the orthogonal projection from $L_2(\bS)$ onto  the closed subspace $\lN:= \cspan\{ e_n :n\in\N_0\}$. Then 
 the Toeplitz algebra $\T\subset B(\lN)$ is generated by the operators $T_f$, $f\in C(\bS)$, where 
 \[ \label{TfN} 
 T_f(\phi):=  P(f\hs \phi), \qquad \phi \in \cspan\{ e_n : n\in \N_0\}  \subset L_2(\bS), 
\]
 and the symbol map $\s : \T \ra C(\bS)$ reads $\s (T_f) = f$. 

 On the topological side, the C*-algebra of continuous  functions on $\bS$ satisfying the conditions on 
 $\hat\s(f)$ in \eqref{CTq} is isomorphic to $C(\underset{k=1}{\overset{2g}{\mbox{$\vee$}}} \bS )$, where 
 where $\underset{k=1}{\overset{2g}{\mbox{$\vee$}}} \bS$  denotes the topological wedge product of $n$ circles. 
Indicating the orientation of the identified arcs with a $+$ or a $-$, we set 
\begin{align*} 
C(\veem \bS ) &:= 
\{ f\in \CS :  f(a_k(t))= f(a_k^{-1}(t)) \ \,\text{for \,all}  \ \,t\in[0,1], \  \,k=1,\ldots,2g\}, \\
C(\veep \bS ) &:= 
\{ f\in \CS :  f(a_k(t))= f(b_k(t)) \ \,\text{for \,all}  \ \,t\in[0,1], \  \,k=1,\ldots,g\} . 
\end{align*} 
Then the definitions of $\CTq$  and $\CPq$ can be rewritten as 
$$
\CTq = \{ f\in\T:  \s(f) \in  C(\veem \bS ) \} ,   \quad 
\CPq = \{ f\in\T:   \s(f) \in  C(\veep \bS ) \}. 
$$

 By the short exact sequence \eqref{Cex1} and the definitions of $\CTq$ and $\CPq$, 
 these C*-al\-ge\-bras contain obviously  $\KlN=\ker(\s)$.  
 Restricting the symbol map to $\CTq$ and $\CPq$ 
 yields the C*-extensions 
\begin{equation} \label{ext2}
\begin{split}
\xymatrix{
 0\;\ar[r]&\; \KlN\;\ar[r] &\; \CTq \;\ar[r]^{\s\ \ } &\;  C(\underset{k=1}{\overset{2g}{\vee}} \bS )\;\ar[r] &\ 0\, ,} \\[-4pt] 
 \xymatrix{
  0\;\ar[r]&\; \KlN\;\ar[r] &\; \CPq \;\ar[r]^{\s\ \ } &\;  C(\underset{k=1}{\overset{g}{\vee}} \bS )\;\ar[r] &\ 0\, .}
  \end{split}
\end{equation}

The C*-algebra extensions \eqref{ext2} provide a computational tool for calculating the $K$-groups of the noncommutative compact surfaces.  
 The result can be found in \cite{W0} and is given by 
 \begin{equation} \label{K}  
 \begin{split}
 K_0(\CTq) \,\cong\, \Z\oplus \Z, \qquad   K_1(\CTq) \,\cong\, \underset{k=1}{\overset{2g}{\oplus}}\Z, \\
 K_0(\CPq)\, \cong\, \Z_2\oplus \Z, \qquad   K_1(\CPq) \,\cong\, \underset{k=1}{\overset{g-1}{\oplus}}\Z. 
   \end{split}
 \end{equation} 
 Moreover, a set of free generators for $K_0(\CTq)$ is given by the trivial projection $[1]\in K_0(\CTq)$ and 
 $[p_{e_0}] \in K_0(K(\lN)) \hookrightarrow K_0(\CTq)$, where $p_{e_0}$ denotes the 1-di\-men\-sional projection onto $\C\hs e_0$. 
 Note that the $K$-groups coincide with those of the classical counterparts but the function algebras  
 $C(\bT_g)$ and $C(\mathbb{P}_g)$ do not contain non-trivial $1{\times} 1$-projections.

\section{Differential geometry of noncommutative compact surfaces} 
 
 \subsection{Spectral triples and regularity}    \label{STR} 
 Since surfaces are even dimensional, we are looking for even spectral triples  $(\A,\hH, D, \g)$ 
 for our noncommutative compact surfaces, 
i.e., a dense *-sub\-algebra $\A$ of $\CTq$ (or $\CPq$) which is stable under holomorphic functional calculus, 
a faithful representation $\pi: \A \ra B(\hH)$, a self-adjoint operator $D$ on $\hH$ with compact resolvent and a self-adjoint 
grading operator $\g$ satisfying $\g^2= \id$, $\g D=-D\g$, $\g\pi(a) =  \pi(a)\g$  and 
$[D,\pi(a)]:= D\pi(a)-\pi(a)D\in B(\hH)$ for all $a\in \A$.

We say that  a spectral triple is $n^+$-summ\-able if $(1+ |D|)^{-(n+\epsilon)}$ yields 
a trace class operator for all $\epsilon >0$ but $(1+ |D|)^{-n}$ does not.   
In this case, one refers to the number $n\in[0,\infty)$ as the metric dimension 
in analogy to Weyl's formula for the asymptotic behavior of the eigenvalues of the Laplacian on a compact 
Riemannian manifold. 

A spectral triple  $(\A,\hH, D)$ is called regular, if $\pi(a),\, [D,\pi(a)])\in \cap_{n\in\N} \hs \dom(\delta^n_{|D|})$ 
for all $a\in \A$, where $\delta_{|D|}(x):= [|D|,x]$   and   $\dom(\delta_{|D|}) :=\{ x\in B(\hH) : \overline{\delta_{|D|}(x)}\in  B(\hH) \}$.  
Let $D=F\hs|D|$ denote the polar decomposition of $D$. 
For regular even spectral triples, one can show that $F=  \begin{pmatrix} 0 & F_{+-} \\ F_{-+}   & 0 \end{pmatrix}$ 
provides an even Fredholm module for $\A$ and one defines $\ind(D):= \ind( F_{+-} )$. 
This Fredholm module is called the fundamental class of $D$  and we say that the  fundamental class is non-trivial 
if it gives rise to non-trivial index pairings. 

Before turning our attention to spectral triples for $\CTq$, we will describe more explicitly the action of $\T\subset B(\lN)$ 
on $\lN$. As in the previous section,  we use the orthonormal  basis $\{ e_k= \frac{1}{\sqrt{2\pi}}u^k:k\in\Z\}$ 
of $L_2(\bS)$.  Then $T_{u^n}=S^n$ and $T_{u^{-n}}=S^{*n}$, where $S\in B(\lN)$ denotes the shift operator given by 
$$
Se_k= e_{k+1}, \qquad k\in\N_0. 
$$
Expanding $f\in \CS\subset L_2(\bS)$ in its Fourier series  $f= \sum_{k\in \Z} f_k u^k$, 
$f_k= \frac{1}{\sqrt{2\pi}} \ip{e_k}{f}$, and setting 
$$
S^{\# k} := S^k, \quad k\geq 0, \qquad S^{\# k} := S^{*|k|}, \quad k<0,
$$
we can write 
\[ \label{Tf} 
T_f= \msum{k\in\Z}{} f_k S^{\# k}, 
\] 
which means that $T_f e_n = \sum_{k\in\N_0} f_{k-n} e_k$. 
Identifying $C^{(m)}(\bS)$ with the space of $m$-times continuously differentiable $2\pi$-periodic functions on $\R$, 
one shows by partial integration that 
\[  \label{Tf'} 
T_{f'}= \im\msum{k\in\Z}{} kf_k S^{\# k}, \qquad f\in C^{(1)}(\bS), 
\] 
where $f' = \frac{\dd}{\dd t} f$ for $t\in [0,2\pi]$. In particular, $f= \sum_{k\in \Z} f_k u^k\in C^{\infty}(\bS)$ 
implies that $\{f_k\}_{k\in\Z}$ is a sequence of rapid decay.

The condition that $\A$ should be stable under holomorphic functional calculus is often ignored 
because the proof that a chosen subalgebra has this property might be somewhat involved. 
Our choice is presented in the next proposition. The proof follows the arguments of \cite[Proposition 1]{C2}.  
\begin{proposition}  \label{P1}
Let $\K_S\subset \KlN$ denote the ideal of matrices of rapid decay, i.e., operators $A$ given by 
$A e_n = \sum_{k\in\N_0} a_{kn}e_k$ such that$\underset{(k,n)\ra \infty}{\lim}\! \! |k^\alpha  a_{kn} n^\beta| =0$ 
for all $\alpha, \beta \in\N_0$. 
Set $C^\infty(\veem \bS ) := C^\infty(\bS) \cap C(\veem\bS ) $ 
and let  $\A$ be the *-algebra generated by the elements of $\K_S$ and 
$\Cinf:=\{ T_f\in \CTq : f\in C^\infty(\veem \bS )\}$.  
Then $\A$ is dense in $\CTq$ and stable under holomorphic functional calculus. 
\end{proposition} 
\begin{proof} 
Note that $\|T_f-T_g\|\leq \|P\|\,\|f-g\|_\infty\leq \|f-g\|_\infty$ and $x = x - T_{\s(x)}  + T_{\s(x)}$ with $x - T_{\s(x)} \in \KlN$. 
Since each $g\in C(\veem\bS )$ can be uniformly approximated by a by functions from $C^\infty(\veem\bS )$, 
and each compact operator can be approximated by matrices of rapid decay, $\A$ is dense in $\CTq$.  

An elementary calculation shows that 
 \begin{align} 
\begin{split} \label{Srel} 
S^{\#(m+k)} - S^{\#m} S^{\#k}  &=  (1- S^mS^{*m}) S^{*(|k|-m)}    (1- S^{|k|} S^{*|k|}),     
 \ \ k\hsp<\hsp 0,\ 0\hsp<\hsp m\hsp <\hsp |k|, \\
S^{\#(m+k)} - S^{\#m} S^{\#k}  &=  (1- S^mS^{*m}) S^{(m-|k|)}    (1- S^{|k|} S^{*|k|}),      \ \ k\hsp<\hsp0,\ m\hsp\geq \hsp |k|, 
\end{split} 
\end{align} 
and 0 otherwise, where the operators in \eqref{Srel} are finite matrices. 
Let $f,g\in C^\infty(\veem \bS )$. Using the representation \eqref{Tf} together with Equation \eqref{Srel}
and the fact that the coefficients are sequences of rapid decay,  
one proves that 
\[  \label{Ks}
T_{fg} - T_f T_g= \msum{m,k\in\Z}{}  f_m g_k  (S^{\#m+k}  - S^{\#m} S^{\#k} )\in \K_S.
\]
Thus $[T_f,T_g]\in\K_S$, so 
the algebra generated by $\Cinf$ is commutative modulo the ideal $\K_S$. 
Since $\K_S$ is an two-sided ideal in $\A$, we get the exact sequence 
$$    
\xymatrix{
 0\;\ar[r]&\; \K_S \;\ar[r] &\; \A \;\ar[r]^{ \s\qquad } &\; C^\infty(\veem \bS )\;\ar[r] &\ 0\, .}  
 $$
Moreover, $\K_S$ and $C^\infty(\veem \bS )$ with the usual (semi-)norms for sequences of rapid decay and $C^\infty$-functions, 
respectively, can be turned into Fr\'echet algebras which are stable under holomorphic functional calculus. To prove the proposition, 
it thus suffices to show that,  for all $a\in \A$ such that $a^{-1} \in \CTq$, one has $a^{-1} \in\A$. 

Now, if $a\in \A$ such that $a^{-1} \in \CTq$, then clearly $\s(a^{-1}) = \s(a)^{-1} \in C^\infty(\veem \bS )$. 
Furthermore, $1- aT_{\s(a^{-1})}  \in \A$ and $\s(1 - aT_{\s(a^{-1})}) =0$, hence $1- aT_{\s(a^{-1})}  \in \K_S$. 
As $\K_S$ is a two-sided ideal in $\CTq$, we obtain $a^{-1} = a^{-1} (1 - aT_{\s(a^{-1})}) + T_{\s(a^{-1})} \in \A$.  
This finishes the proof. 
\end{proof} 

As also in the classical case not all continuous functions are differentiable, the last proposition provides us with a preferred choice 
of a dense *-subalgebra of  $\CTq$ for the construction of a spectral triple. In \cite{WD}, a spectral triple for the 
noncommutative torus $C(\mathbb{T}_{1,q})$ 
was obtained by starting with the action of the first order differential operator $\frac{\partial}{\partial z}$ 
on the Bergman space $\AD$. It turns out that essentially the same Dirac operator works for all the 
noncommutative compact surfaces from Section \ref{sec-1}. In the following, we will first give the definition of 
these spectral triples and then prove their fundamental properties in a theorem. 

The simplest choice of a $\Z_2$-graded Hilbert space with a faithful representation of $\CTq\hsp\subset \hsp B(\lN)$ is 
$\hH\hsp:=\hsp \lN\oplus\lN$. Then the assignment  
$\pi : \A \ra B(\lN\oplus\lN)$,  $\pi(a)(v_+\oplus v_-):= av_+\oplus av_-$, evidently defines a faithful *-re\-pre\-sen\-ta\-tion 
for all *-sub\-al\-ge\-bras $\A\subset B(\lN)$.  The representation commutes with  obvious grading operator $\g$ given by 
$\g(v_+\oplus v_-)= v_+\oplus (-v_-)$. 
Consider now  the self-adjoint  number operator $N$ defined by 
 \[   \label{N}
 N e_n = n e_n, \quad \dom(N) = \Big\{ \msum{n\in\N_0}{} \a_n e_n\in\lN:  \msum{n\in\N_0}{} n^2 |\a_n|^2<\infty\Big\} , 
 \] 
 and let 
 \[    \label{D}
 D:=  \begin{pmatrix} 0 & \!\! S^* N \\ NS  & 0 \end{pmatrix}=  \begin{pmatrix} 0 &  \!\!  S^* N \\ S (N+1)  & 0 \end{pmatrix}, 
 \quad \dom(D) = \dom(N) \oplus \dom(N) \subset \hH. 
 \]
Since $ (S^*N)^*=NS = S (N+1)$, $D$ is self-adjoint. Defining $F$ and $| D |$ by the polar decomposition $D= F\hs |D|$, we get 
from \eqref{D} 
  \[   \label{FD}
 F=  \begin{pmatrix} 0 & S^* \\ S   & 0 \end{pmatrix} , \qquad | D |=  \begin{pmatrix} N+1 & 0 \\ 0   &  N\end{pmatrix} . 
 \]  
 Clearly, $\g D=-D\g$.  
 
 \begin{theorem}  \label{TST}
 Let $\A \subset \CTq$ denote the pre-C*-algebra from Proposition \ref{P1}, and let $\hH$, $D$, $\g$ and the 
 representation $\pi:\A\ra B(\hH)$ be given as in the previous paragraph.  
 Then $(\A,\hH,D,\g)$ yields a $1^+$-summable regular even spectral triple for $\CTq$. The Dirac operator $D$ has 
 discrete spectrum $\spec(D) = \Z$, each eigenvalue $k\in \Z$ has multiplicity 1 and a complete orthonormal basis 
 of associated eigenvectors is given by 
 \[   \label{bk} 
b_k:= \mbox{$\frac{1}{\sqrt{2}}$}(e_{k-1} \oplus e_k), \quad 
b_{-k}:= \mbox{$\frac{1}{\sqrt{2}}$}(- e_{k-1} \oplus e_k), \quad  k>0, \qquad b_0:= 0\oplus e_0. 
\]
 Moreover, the fundamental class of $D$  is non-trivial.  % in the sense that  $\ind(D) =\ind(S^*)= 1$. 
 \end{theorem} 
 
 \begin{proof} 
 It was already discussed above  that $\g$, $D$ and $\pi(a)$, $a\in\A$, satisfy the commutation relations of a $\Z_2$-graded  spectral triple. 
 To prove that commutators the $[D,\pi(a)]$ are bounded, it suffices to consider the generators from $\K_S$ and $\Cinf$ since, 
 by the Leibniz rule for commutators, $[D,\pi(ab)]= [D,\pi(a)]\pi(b) + \pi(a)[D,\pi(b)]$. 
 
 If the operator $A \in\K_S$ is given by a matrix of rapid decay  $(a_{kn})_{k,n\in\N_0}$, then 
 the matrices $(n a_{kn})_{k,n\in\N_0}  $ and  $(k a_{kn})_{k,n\in\N_0} $ corresponding to $NA$ and $AN$, respectively, are again matrices 
 of rapid decay and thus define bounded operators  belonging to $\K_S$. Using the relations $S^*N = (N+1)S^*$ and 
 $S(N+1)=NS$, we get 
 %SN=(N-1)S= NS -S   NS = SN + S  S* N = NS* +S*
 \[   \label{DpA} 
 [D,\pi(A)]  =  \begin{pmatrix} 0 & S^*N  A - A N S^* - AS^*\\ SNA-A N S +SA & 0 \end{pmatrix} 
 \]
which  is bounded because all the operators $A$, $S$, $S^*$, $AN$ and $NA$ belong to $B(\lN)$.  
We can conclude even more, namely that the entries of $[D,A]$ belong to $\K_S$ since, as easily seen, 
the product of a matrix of rapid decay with the shift operator or its adjoint yields again a matrix of rapid decay. 

Next, let $f\in C^\infty(\veem \bS )$ so that $T_f \in \Cinf$. Then, by Equations \eqref{Tf} and \eqref{Tf'}, 
\begin{align}  \nonumber
[S^*N, T_f] e_n &=  \msum{k\in\N_0}{}  k f_{k-n} e_{k-1}  -  \msum{k\in\N_0}{} n f_{k-n+1} e_{k} 
=  \msum{k\in\N_0}{} (k-n+1) f_{k-n+1 } e_{k} \\
&= \msum{k\in\Z}{} kf_k S^{\# k+1} e_n = -\im T_{\bar u f'}e_n,  \label{SNTf} 
\end{align}
hence 
\[   \label{SNT} 
[S^*N, T_f]= -\im T_{\bar u f'} \in \T, \qquad [NS, T_f] = (- [ S^*N, T_{\bar f}])^* =   -\im T_{u f'}\in \T. 
\] 
Therefore $[ D, \pi(T_f)] \in B(\hH)$ since the entries of this $2\times 2$-matrix belong to the Toeplitz algebra 
$\T\subset B(\lN)$. 
This finishes the proof that $[D,\pi(a)]\in B(\hH)$ for  all $a\in\A$. 

Furthermore, one readily verifies that the vectors $b_k$, $k\in \Z$, in Equation \eqref{bk} form a basis of eigenvectors for $\hH$ 
and that $Db_k=kb_k$. Since each eigenvalue $k\in\Z=\spec(D)$  has multiplicity 1,  the resolvent $(D+\im)^{-1}$ is compact and  
the spectral triple is $1^+$-summable. 

To prove the regularity of the spectral triple, first note that 
$$
[|D|, A] = \begin{pmatrix} [ N , A_{11}] &\!\!\!\! [ N , A_{12}] + A_{12}\\ [ N , A_{21}] - A_{21} & [ N , A_{22}] \end{pmatrix}, \ \ 
A=  \begin{pmatrix} A_{11} &\!\!\!\!   A_{12}\\  A_{21} & A_{22} \end{pmatrix}\in B(\lN \oplus \lN). 
$$
Hence it suffice to show that all elements $a\in \A$ and all entries of $[D,\pi(a)]$ 
belong to $\cap_{n\in\N} \hs \dom(\delta^n_{N})$, where $\delta_{N}(x):= [N,x]$ for $x\in B(\lN)$. 

If $A\in\K_S$ is given by a rapid decay matrix, then $N^n A \in \K_s$ and  $A N^n  \in \K_s$ for all $n\in\N$ by the 
very definition of a rapid decay matrix. This implies $\K_s\subset \cap_{n\in\N} \hs \dom(\delta^n_{N})$. 
As concluded below Equation \eqref{DpA}, all entries of $[D,\pi(A)]$ belong again to $\K_S$. 
Therefore we have $\pi(A), \,[D,\pi(A)] \in \cap_{n\in\N} \hs \dom(\delta^n_{|D|})$ for all $A\in\K_S$. 

Let now $f\in C^\infty(\bS)$. As in \eqref{SNTf}, we compute $[N,T_f] e_n = \msum{k\in\N_0}{} (k-n) f_{k-n } e_{k}$, 
thus 
\[    \label{NnTf}
 \delta_{N}^n(T_f) = (-\im)^n T_{f^{(n)}} 
\]
by Equation \eqref{Tf'}. Consequently,  $T_f \in \cap_{n\in\N} \hs \dom(\delta^n_{N})$. % for all $f\in C^\infty(\bS)$.  
As $ u f'$ and $\bar u f'$  also belong to $C^\infty(\bS)$, we conclude from  the definition of $\Cinf$ 
together with  Equations  \eqref{SNT}  and \eqref{NnTf} 
that $\pi(T_f) ,  \,[D,\pi(T_f)] \in \cap_{n\in\N} \hs \dom(\delta^n_{|D|})$ for all $T_f \in \Cinf$. 
Since the elements from $\K_S$ and $\Cinf$ generate $\A$, the regularity follows from the 
Leibniz rule for commutators. 

Having a regular, even spectral triple, we know that its fundamental class $F$ from \eqref{FD} defines an even Fredholm module. 
Since $\ind(D)=\ind(S^*)=1\neq 0$, the  fundamental class is non-trivial. 
\end{proof} 
  
 The restriction to $\CTq$ was made  only for notational convenience. Clearly, we have $\K_S\subset \CPq$ without any modification. 
  Analyzing the proofs of Proposition \ref{P1} and Theorem \ref{TST}, one easily realizes that only the differentiability of 
  the function $f$ in the definition of $T_f\in \Cinf$ was used. So, if we define 
  \[
   C^\infty(\mathbb{P}_{g,q}) :=\{ T_f\in \CPq : f\in C(\veep \bS )\cup C^\infty(\bS) \}, 
  \] 
 then all arguments in the proofs remain valid. Therefore we can state the analogous results of Theorem \ref{TST} for $\CPq$. 
  
  \begin{corollary}
  Let $\A\subset \CPq$ denote  the *-algebra generated by the elements of $\K_S$ and $C^\infty(\mathbb{P}_{g,q})$. 
  With the same definitions of $\hH$, $D$, $\g$  and $\pi:\A\ra B(\hH)$ as in Theo\-rem~\ref{TST}, 
  $(\A, D,\hH, \g)$ yields a $1^+$-summable regular even spectral triple for $\CPq$.  
  The spectrum, the eigenvalues and the orthonormal basis of eigenvectors from 
  Theorem~\ref{TST} remain unchanged for the Dirac operator $D$ and its fundamental class is non-trivial. 
  \end{corollary}
  
  At the first glance, it might be surprising that there exist Dirac operators on noncommutative versions of a non-orientable manifolds. 
  However, the Dirac operator of a spectral triple should rather be viewed as an analog of an elliptic first order differential operator 
  and not necessarily as the Dirac operator on a spin manifold. Nevertheless, since the Dirac operator is related to topological invariants 
  via the index theorem, it is a strange effect of noncommutative geometry that the same operator can be used for 
  different noncommutative spaces among which some have $K$-groups with torsion. 
  
 \subsection{Real structure and first order condition}    \label{rs} 
 
 In the context of spin geometry, a real structure singles out those manifolds that admit a real spin structure. 
 The results below will show that our spectral triples cannot be equipped with a real structure in the exact sense.  
 Since this result would not be surprising for a quantum space 
 for which the classical counterpart does not 
 admit a real spin structure, we restrict ourselves in this section to the quantized orientable 
 surfaces $\CTq$. 
 
A real structure for a spectral triple $(\A, \hH, D)$ is given by an anti-unitary operator $J$ satisfying 
$J^2 = \pm\hs \id$ and $JD=\pm DJ$. For an even spectral triple  $(\A, \hH, D,\g)$, 
one requires additionally $J\g = \pm \g J$. The signs depend on the dimension of the underlying 
quantum space. For instance, 
\[    \label{J2}
J^2 = -\hs \id, \qquad JD=  DJ, \qquad J\g =  - \g J
\] 
in dimension 2, and 
\[    \label{J1}
J^2 =  \id,  \qquad JD= - DJ 
\] 
in dimension 1. 

Given a real structure $J$, one says that the spectral triple satisfies the first order condition if for all $a,b\in\A$
\[  \label{J0} 
[\pi(a), J\pi(b)J^{-1}] = 0,  \quad [[D,\pi(a)], J\pi(b)J^{-1}] =0\quad \text{for \,all}\ \,a,b\in\A.
\]
It was observed in \cite{DLPS} and \cite{DLSSV} that, under certain circumstances, a real structure might not exist 
for quantized real spin manifolds and it was proposed to modify the first order condition \eqref{J0} by requiring only 
\[  \label{JK} 
[\pi(a), J\pi(b)J^{-1}] \in \K_S(\hH) ,  \quad [[D,\pi(a)], J\pi(b)J^{-1}] \in \K_S(\hH)  \quad \text{for \,all}\ \,a,b\in\A, 
\]
where $\K_S(\hH)  \subset B(\hH)$ denotes the ideal of matrices of rapid decay 
(associated to an orthonormal basis). Here we assume that $\hH$ is separable. 
In the context of noncommutative geometry, the matrices of  rapid decay  are considered as infinitesimals of arbitrary 
high order. 

The first result of this section shows that the spectral triples from Theorem \ref{TST} do not admit a real 
structure in the exact sense.  
 
\begin{proposition}  \label{PJ}
Let $(\A, \hH, D,\g)$ denote the spectral triple described in Theorem \ref{TST}. 
Then there does not exist an anti-unitary operator $J$ on $\hH$ satisfying \eqref{J2}. 
\end{proposition} 
 \begin{proof} 
From $J^{-1} D J e_k = D e_k = ke_k$,   it follows that $D J e_k =k J e_k$, hence $J e_k = \alpha_k e_k$, where $\alpha_k\in \C$. 
Since $J$ is anti-unitary,  we have necessarily $|\alpha_k|=1$, and thus 
$$
J^2 e_n = J(\alpha_k e_k) = \bar\alpha_k J e_k = \bar\alpha_k \alpha_k e_k =e_k, 
$$ 
which contradicts the first equation of \eqref{J2}. 
 \end{proof} 

Ignoring for a moment the commutation relations \eqref{J2} and \eqref{J1}, it is also impossible to find an anti-unitary operator 
satisfying \eqref{J0}.  

\begin{proposition}  \label{JpiJ}
 Let $\A$ denote the *-algebra from Proposition \ref{P1} and 
 $\pi : \A \ra B(\hH)$  the representation  from Theorem \ref{TST}. 
 Then there does not exist a anti-unitary operator $J$ on $\hH$ satisfying 
 the first order condition \eqref{J0}. 
\end{proposition} 
 \begin{proof} 
Recall that $\hH= \lN\oplus \lN$ and $\pi(a)(v_+\oplus v_-):= av_+\oplus av_-$. 
If an operator $A=  \begin{pmatrix} A_{11} & A_{12}\\  A_{21} & A_{22} \end{pmatrix}\in B(\lN \oplus \lN)$ 
commutes with  $\pi(a)$ for all  $a\in\A$, then we have necessarily 
$[k, A_{ij}]=0$ for all $k\in\K_S$ and $i,j=1,2$. Since $\K_S$ is dense in $K(\lN)$, it follows 
that $A_{ij} = c_{ij}\,\id$ with $c_{ij}\in\C$. 
 Therefore we get for any anti-unitary operator $J$ satisfying \eqref{J0} 
 $$
 \dim\{ A\in B(\hH): [ \pi(a), A] =0 \text{ for all }  a\in\A\} =4\geq  \dim\{ J\pi(b)J^{-1}:  b\in\A\} = \infty, 
 $$  
which is a contradiction. 
 \end{proof}  
 
If one wants to allow for the existence of a real structure despite the negative result of Proposition \ref{JpiJ}, 
one may consider to weaken the first order condition \eqref{J0} by requiring only the modified version \eqref{JK}.  
But then the problem of Proposition \ref{PJ} still persists. However, the problem was caused by 
taking the commutation relations \eqref{J2} of a  spectral triple of dimension 2.  
On the other hand, our spectral triples are $1^+$-summable, thus their metric dimension is 1 instead of 2. 
%Such dimension drop phenomena were observed before, for instance in  \cite{DS,MNW1,MNW2}.  
Under the requirements of Equations \eqref{J1} and \eqref{JK}, we can prove the following positive result. 

\begin{proposition}  \label{JpiJK}
Let $(\A, \hH, D,\g)$ denote the spectral triple given in Theorem \ref{TST}, 
and let $\{b_k:k\in\Z\}$  be the orthonormal basis defined in \eqref{bk}. 
Then the anti-unitary operator $J$ given by 
\[    \label{Jbk}
J b_k= b_{-k} , \qquad k\in\Z, 
\] 
satisfies the conditions \eqref{J1} and \eqref{JK}. 
\end{proposition} 
 \begin{proof} 
 Since $b_k$ is an eigenvector of $D$ corresponding to the eigenvalue $k\in\Z$, we have 
 $JDJ^{-1} b_k = -k b_k = -D b_k$ , thus $JD= -DJ$. 
 Obviously, $J^2=1$, so \eqref{J1} is satisfied. 
 
 If $b\in \K_S$, then $J\pi(b)J^{-1}\in\K_S(\hH)$ and therefore \eqref{JK} holds for all $a\in\A$ 
 since $\K_S(\hH)$ is an ideal in $B(\hH)$. Similarly, if $a\in \K_S$, then, as in the proof of Theorem \ref{TST}, 
 $Na, \,aN\in\K_S$. Therefore $[D,\pi(a)] \in\K_S(\hH)$, and since also $\pi(a) \in\K_S(\hH)$, the condition 
  \eqref{JK} is now fulfilled for all $b\in\A$. 
 
 It remains to show that \eqref{JK} holds  for $a,b\in \Cinf$. So, let $a=T_g$ and $b= T_f$, 
 where $g,f\in C^\infty(\veem \bS )$.  From \eqref{bk}  and  \eqref{Jbk}, we get 
 $J(e_k\oplus e_n) = (-e_k)\oplus e_n$ for all $k,n\in\N_0$. 
 Writing $T_f= \sum_{k\in\Z}f_k S^{\# k}$ as in \eqref{Tf} and setting 
 $e_n^+:= e_n \oplus 0$ and $e_n^-:= 0\oplus e_n$, we compute 
\[     \label{JTfJ}
 J \pi(T_f) J^{-1}  e_n^\pm    =  \msum{k\in\N_0}{} \bar f_{k-n} e_k^\pm = \sum_{k\in\Z}\bar f_k S^{\# k} e_n^\pm  = \pi(T_{\hat f}) e_n^\pm , 
 \]
where $\hat f\in C^\infty(\bS)$ is defined by $\hat f(z):= \overline{ f(\bar z)}$. 
Therefore, by \eqref{D},  \eqref{SNT} and \eqref{JTfJ}, 
$$
[\pi(T_g), J\pi(T_f)J^{-1}]  = \pi([T_g, T_{\hat f}]),  \ \  
[[D,\pi(T_g)], J\pi(T_f)J^{-1}] =  -\im  \begin{pmatrix} 0 & \!\!\!\![ T_{\bar u g'},T_{\hat f}]   \\  [ T_{ u g'},T_{\hat f}]    & 0 \end{pmatrix} . 
$$
Since $g, \hat f,  ug', \bar u g' \in C^\infty(\bS)$, 
and hence its Fourier coefficients are sequences of rapid decay, 
we conclude as in the paragraph with Equation \eqref{Ks} 
that the commutators $[T_g, T_{\hat f}]$, $[ T_{\bar u g'},T_{\hat f}]$ and $[ T_{ u g'},T_{\hat f}]$ belong to $\K_S$. 
From this, the result follows. 
 \end{proof} 

\subsection{Finiteness}  

We say that a spectral triple $(\A,\hH,D)$ satisfies the finiteness condition, if there exists a dense subset 
of ``smooth'' vectors $\hH_\infty \subset \cap_{n\in\N} \,\dom(D^n)$ which is isomorphic to a finitely generated 
left $\A$-projective module. This finitely generated projective module is considered as a module of smooth sections 
of the vector bundle on which the Dirac operator acts (e.g.\ spin bundle). 
If $\hH_\infty$ is a core for the self-adjoint operator $D$, then the Dirac operator $D$ is uniquely determined by its 
restriction to $\hH_\infty$. 

Stable isomorphism classes of finitely generated projective modules of a (pre-)C*-al\-ge\-bra 
are classified by the associated $K_0$-group.  
Recall from Section \ref{sec-1}  that a non-trivial generator of $K_0(\CTq)$ 
is given by the 1-dimensional projection $p_{e_0} = 1-SS^*\in \K_S\subset \A$. 
Consider now the finitely generated left $\A$-projective module $\A_\infty := \A(1-SS^*) \subset B(\lN)$ 
together with the vector state and the non-negative sesquilinear form 
\[    \label{state} 
\psi_0(a) := \ip{e_0}{a e_0}, \quad \ipp{b}{a} := \psi_0(b^*a) = \ip{be_0}{a e_0} \quad a,b \in  \A(1-SS^*) . 
\]
The next proposition shows that $\cap_{n\in\N} \,\dom(D^n) \cong  \A(1-SS^*) \oplus  \A(1-SS^*) $ 
satisfying thus the finiteness condition. We will give the proof for $\CTq$ but the statement holds also for 
$\A\subset \CPq$. 

\begin{proposition}
Consider the finitely generated left $\A$-module $\A(1-SS^*)$ together with the sesquilinear form 
given in \eqref{state}. Then the left $\A$-modules $\hH_\infty := \cap_{n\in\N} \,\dom(D^n)$ 
and $ \A(1-SS^*) \oplus  \A(1-SS^*)$ 
are isometrically isomorphic. 
\end{proposition} 
\begin{proof}
Since $\cap_{n\in\N} \,\dom(D^n) = \cap_{n\in\N} \,\dom(N^n) \oplus \cap_{n\in\N} \,\dom(N^n) $ with the diagonal $\A$-action, 
it suffices to show that $\A(1-SS^*) \cong  \cap_{n\in\N} \,\dom(N^n)$. Note that 
\begin{align} \nonumber 
\underset{n\in\N}{\cap} \,\dom(N^n)  &= \{ \msum{k\in\N_0}{} \a_k \hs \e_k\in\lN :  \msum{n\in\N_0}{} |\a_k|^2  k^{2n} <\infty\} \\
&= \{ \msum{k\in\N_0}{} \a_k \hs \e_k\in\lN :  (\a_k)_{k\in\N_0}  \in \cS(\N)  \}         \label{S}
\end{align} 
where $ \cS(\N) $ denotes the space of sequences of rapid decay. We claim that 
\[  \label{Phi}
\Phi :  \A(1-SS^*) \subset B(\lN)\, \lra\, \underset{n\in\N}{\cap} \,\dom(N^n)  , \qquad \Phi(x):= x\hs e_0 , 
\] 
defines an isometric isomorphism. 

To see that $\Phi$ is well defined, suppose that $T_f \in \Cinf$, 
where $f\in C^\infty(\veem \bS )$, and write $f= \sum_{k\in \Z} f_k u^k$ in its Fourier series expansion. 
From Equation \eqref{Tf}, we get  $\Phi(T_f(1-SS^*)) = \sum_{k\in \N_0} f_k\hs e_k$. Since $f\in C^\infty(\bS)$, the sequence 
$(f_k)_{n\in\N_0}$ belongs to $\cS(\N)$ (cf.\ Section \ref{STR}). Therefore, by \eqref{S}, 
$\Phi(T_f(1-SS^*))  \in \cap_{n\in\N} \,\dom(N^n)$. If $a\in \K_S$ and $(a_{kj})_{k,j\in\N_0}$ denotes 
the corresponding matrix of rapid decay, then $(a_{k0})_{k\in\N_0}\in \cS(\N)$ and therefore 
 $\Phi(a(1-SS^*)) = \sum_{k\in \N_0} \! a_{k0}\hs e_k\in \cap_{n\in\N} \,\dom(N^n)$. 
Since each element from $a\in \A$ can be written  as $a = (a- T_{\s(a)}) + T_{\s(a)} \in \K_S+ \Cinf$
it follows that  $\Phi(\A(1-SS^*)) \subset  \cap_{n\in\N} \,\dom(N^n)$.  
  
Clearly, $\Phi$ is left $\A$-linear since $\A\subset B(\lN)$. To show the surjectivity of $\Phi$, 
let $v=\sum_{k\in\N_0} \a_k \hs \e_k\in  \cap_{n\in\N} \,\dom(N^n)$. Then $a := (\a_k \delta_{0,j})_{k,j\in\N_0}\in\K_S$, 
where  $\delta_{i,j}$ denotes the Kronecker delta, and $\Phi(a(1-SS^*)) = v$. 
Moreover, $\Phi$ is injective. To see this, assume that $\Phi(a(1-SS^*)) = 0$. 
Then $a(1-SS^*)e_0= \Phi(a(1-SS^*)) = 0$ and $a(1-SS^*)e_k=0$ for all $k\neq 0$ 
since $1-SS^*$ is the orthogonal  projection onto $\C e_0$. Thus  $a(1-SS^*)=0$ in $\A(1-SS^*) \subset B(\lN)$. 
So we just proved that the map $\Phi$ in \eqref{Phi}  defines an isomor\-phism of left $\A$-modules. 

It remains to show that $\Phi$ is actually an isometry. From \eqref{state} and \eqref{Phi},  we get 
$$
\ip{ \Phi(a(1-SS^*)) }{ \Phi(a(1-SS^*)) } =\ip{ae_0}{a e_0} = \ipp{a}{a}
$$
for all $a\in\A$. 
This completes the proof. 
\end{proof} 

Applying the previous proposition to the noncommutative torus $C(\mathbb{T}_{1,q})$, we see that 
the ``spin bundle'' is given by $ \A(1-SS^*) \oplus  \A(1-SS^*) $ rather than $\A \oplus  \A$ as it should have been in
analogy to the classical case.  Moreover, by  extending the isomorphism $\Phi$ in \eqref{Phi} to its closure, the Hilbert space 
 $ \overline{\A(1-SS^*)} \cong \lN \subset L_2(\bS)$ looks rather like functions on the circle $\mathbb{S}$ than on the disc $\D$. 
 This observation is in line with the metric dimension 1. It seem that the interior of the quantum disc 
 $K(\lN)= \ker(\s)$ has the dimension of a (fuzzy) point. Also Equation \eqref{SNT} indicates that the action 
 of the Dirac operator is essentially given by a derivation on the circle.

\subsection{Existence of a volume form or orientation}  

By a volume form we mean a Hochschild $n$-cycle~$\omega$, i.e., 
\begin{align*}
\omega&=\msum{j}{} a_{0j} \ot b_j \ot a_{1j} \ot \cdots \ot a_{nj} \in \A\ot\A^{\mathrm{op}}\ot\A\ot \cdots \ot \A, \\
0&=\delta_n(\omega):=  \msum{j}{} a_{0j} \ot (b_j a_{1j}) \ot a_{2j}\ot \cdots \ot a_{nj} \\
&+ \msum{k=1}{n-1} \msum{j}{}(-1)^k a_{0j} \ot b_j\ot \cdots \ot a_{kj} a_{k+1,j}  \ot \cdots \ot a_{nj} 
\pm \msum{j}{} a_{nj} a_{0j} \ot b_j \ot  \cdots \ot a_{n-1,j} , 
\end{align*} 
satisfying 
 \[      \label{vol} 
 \g = \pi_D(\omega) := \mbox{$\sum_j$}\, a_{0j} J b_j^*J^{-1} [D, a_{1j}] \cdots [D, a_{nj}] , 
 \] 
 where $n$ depends on the (metric) dimension of the spectral triple. 
 
 As pointed out in the paragraph preceding Proposition \ref{JpiJK}, the metric dimension of our spectral triples is 1. 
 However, for a Hochschild $1$-cycle~$\omega$, the expression $\pi_D(\omega)$ in \eqref{vol} is an odd operator whereas  
 $\g$ is a diagonal operator, so Equation \eqref{vol} cannot hold. Assuming that our quantum surfaces have dimension 2, 
 we face the problem of the non-existence of the real structure $J$, see Proposition \ref{PJ}. 
 Unfortunately the problem goes deeper and cannot be resolved in any other way. 
 
 \begin{proposition}   \label{O}
 For all spectral triples from Section \ref{STR} and any anti-unitary operator $J$, there does not exist a Hochschild $n$-cycle 
 satisfying \eqref{vol}. 
  \end{proposition} 
 \begin{proof} 
 For $n$ odd, the same reasoning as in the case $n=1$ applies: Since $D$ is odd, the right hand side of \eqref{vol} would yield an odd 
 operator whereas $\g$ is even, a contradiction. 
 
 Assume now that there exist an anti-unitary operator $J$ and a Hochschild $2k$-cycle 
 $\omega=\sum_{j}a_{0j} \ot b_j \ot a_{1j} \ot \cdots \ot a_{2k,j}$ satisfying \eqref{vol}.  
 Let $\hat \s : B(\lN) \ra B(\lN) / K(\lN) $ denote the cannonical projection and observe that the restriction of $\hat \s$ to the Toeplitz algebra yields the symbol map. Recall that $Nk, kN\in\K_S$ for all $k\in\K_S$. 
 From \eqref{SNT}, we get for all $a\in\A$ 
 $$
 \hat\s ([S^*N, a] ) = \hat\s ([S^*N, a - T_{\s(a)} + T_{\s(a)}  ] )  =  \hat\s ([S^*N, T_{\s(a)}  ] ) 
 = \hat\s ( -\im T_{\bar u \s(a)'}) = -\im \bar u \s(a)', 
 $$
 and similarly $\hat\s ([NS, a] ) =  -\im u \s(a)'$.  
 Using the facts that 
 the operators $\s(t)$, $t\in\T$, commute and that
 $u\bar u=1$, and applying $\hat \s$ to two consecutive commutators in \eqref{vol},  we obtain 
 \begin{align*}
 \hat\s( [D, a_{2i,j}]  [D, a_{2i+1,j}]) & = 
 \begin{pmatrix}   \hat\s([S^* N, a_{2i,j}] [NS, a_{2i+1,j}] )& 0  \\ 
 0 &\!\!\! \hat\s([NS, a_{2i,j}] [S^*N, a_{2i+1,j}]) \end{pmatrix}\\
 &=  \begin{pmatrix}   - \s(a_{2i,j})'  \s(a_{2i+1,j})'  & 0\\ 0 &  - \s(a_{2i,j})'  \s(a_{2i+1,j})'   \end{pmatrix}
 \end{align*}
 Note that the diagonal entries coincide. This remains true for the right hand side of \eqref{vol} 
 if $n$ is even. On the other hand, the diagonal entries of $\g$ differ by a minus sign. 
 Therefore applying $\hat \s$ to the diagonal elements of \eqref{vol} and equating gives 
$$
1= \msum{j}{}  (-1)^k \hat\s( a_{0j} J b_j^*J^{-1}) \s(a_{1,j})'\s(a_{2,j})'\cdots \s(a_{2k-1,j})' \s(a_{2k,j})' =-1,   
$$
 which is a contradiction.  
 \end{proof} 
 
 Observe that the problem of Proposition \ref{O} cannot be solved by changing $\g$ as in \cite{W} as long as 
 the diagonal entries of $\g$ do not coincide.

  \subsection{Poincar\'e duality}  
  
We mentioned already in Section \ref{STR} that regular spectral triples give rise to index pairings. 
 To be more precise, one can show that, for an even regular spectral triple $(\A, \hH,D,\g)$ 
 and a projection $P=P^2=P^*\in\Mat_{n\times n}(\A) $, 
 the map $\pi(P)D_{+-}\pi(P) : \pi(P)\hH^n_- \ra  \pi(P)\hH^n_+$ yields a 
 Fredholm operator and its index does not depend on the $K_0$-class of $P$. 
 Here, $D_{+-}$ denotes the upper right corner of the odd operator $D$. 
 Unfortunately, the formulation of Poincar\'e duality involves a real structure $J$. 
 One says that an even spectral triple with real structure $J$ satisfies Poincar\'e duality~\cite{C0}, if 
\[    \label{i} 
K_0(\A)\times K_0(\A)\, \ni \,([P],[Q])\, \longmapsto\, \ind ((\pi(P)\ot J\pi(Q)J^{-1} )D_{+-} (\pi(P)\ot J\pi(Q)J^{-1} )) \,\in\, \Z
\] 
defines a non-degenerate pairing, where $P\in\Mat_{n\times n}(\A) $ and $Q\in \Mat_{k\times k}(\A) $ are projections
and the operator in \eqref{i} acts between the Hilbert spaces 
$(\pi(P)\ot J\pi(Q)J^{-1} )\hH_-^{nk}$ and $(\pi(P)\ot J\pi(Q)J^{-1} )\hH_+^{nk}$. 
 
The aim of this section is to explain that the Poincar\'e duality fails for our spectral triple. 
The statement doesn't seem to make much sense because of the non-existence of a real structure 
proven in Section \ref{rs}. However, the main problem is not the real structure. 
The next proposition will show that the fundamental class of the Dirac operator only detects the 
rank of the trivial $K_0$-classes and leads to a zero pairing with $K_0$-classes represented by 
compact operators. Let us recall here from \cite{W0} (see also the end Section~\ref{sec-1}) 
that $ K_0(\CTq) \cong \Z\oplus \Z \cong K_0(K(\lN)) \oplus K_0(\C)$.  
\begin{proposition}
Let $(\A,\hH,D,\g)$ denote the spectral triple from Theorem \ref{TST}. 
For any odd or even anti-unitary operator $J$ on $\hH= \lN\oplus\lN$, 
the index pairing \eqref{i} (if well defined) is degenerate. 
\end{proposition} 
\begin{proof}
Let $p_{e_0}$ denote the projection 1-di\-men\-sional projection onto $\C\hs e_0\subset \hH_-=\lN$. 
Then $[p_{e_0}] \in K_0(\CTq)$,  and for any projection $Q := \pi(Q_0)$ with $[Q_0]\in K_0(\CTq)$, we have 
$$
\dim((p_{e_0} \ot J QJ^{-1} )\hH_-^n) ) = \dim (J QJ^{-1} ( (p_{e_0}\lN) \ot \C^n)) \leq \dim (\C\ot\C^n) =n<\infty . 
$$
Hence the operator 
$$
(p_{e_0} \ot JQJ^{-1} ) S^*N (p_{e_0} \ot JQJ^{-1} )  : (p_{e_0} \ot JQJ^{-1} )\lN^n \ra (p_{e_0} \ot JQJ^{-1} )\lN^n 
$$
acts between finite dimensional spaces and therefore its index is always 0. 
\end{proof}   
  
Note that the problem arises because the compact operators $K(\lN)\subset\CTq$ act on $\hH_-=\lN$ by the identity 
and thus $\dim(p_{e_0}\hH_-)=1<\infty$. 
So we arrive again at the conclusion that the Hilbert space $\hH=\lN\oplus \lN$ is ``too small'' for a non-trivial 
index pairing.

\subsection*{Acknowledgment} 
This work was supported by CIC-UMSNH and 
is part of the international project ``Quantum Dynamics'' supported by EU grant H2020-MSCA-RISE-2015-691246 
and co-financed by Polish Government grant 3542/H2020/2016/2 awarded for the years 2016-2019. 

%% **********Bibliography***********

%% Use the widest label as parameter.

%% Arrange the reference items in alphabetical order.

%% Reference items can be numbered or have labels of your choice, as below.

%% Abbreviations of journal names should follow Mathematical Reviews.

%% Only the title is italicized; boldface is not used.

%% Our software will add links to many articles; for this, enclosing volume numbers in { } is helpful

%% Do not give the issue number unless the issues are paginated separately.

\end{document}